\documentclass[12pt]{article}
\usepackage{graphicx}
\usepackage{subfigure}
\usepackage{amssymb,amsmath,amsthm,epsfig}
\usepackage{latexsym, enumerate}
\usepackage{eepic}
\usepackage{epic}
\usepackage{graphicx}
\usepackage{color}
\usepackage{ifpdf}
\usepackage{booktabs}

\usepackage{dsfont}
\usepackage{multirow}
\usepackage{bm}

\theoremstyle{plain}
\newtheorem{lem}{Lemma}[section]
\newtheorem{thm}[lem]{Theorem}
\newtheorem{cor}[lem]{Corollary}

\theoremstyle{definition}

\theoremstyle{remark}
\newtheorem{rem}{Remark}[section]

\begin{document}
\title{ \large\bf On uniqueness and nonuniqueness for potential reconstruction in quantum fields from one measurement II. the non-radial case}

\author{
Zhi-Qiang Miao\thanks{College of Mathematics and Econometrics, Hunan University, Changsha 410082, Hunan Province, China. Email: zhiqiang\_miao@hnu.edu.cn}
\and
Guang-Hui Zheng\thanks{Corresponding author. College of Mathematics and Econometrics, Hunan University, Changsha 410082, Hunan Province, China. Email: zhenggh2012@hnu.edu.cn}
}

\date{}
\maketitle

\begin{abstract}
  In this article we study uniqueness and nonuniqueness for potential reconstruction from one boundary measurement in quantum fields, associated with the steady state Schr\"{o}dinger equation. It is an extension of our recent work \cite{Zheng2019}. Based the theory of the ND map and modified bessel function, the uniqueness theorem of the inverse problem in two-dimensional nd three-dimensional core-shell structure is established, respectively. When different potential and shape are considered, the nonuniqueness results is also proved.
\end{abstract}\smallskip

\smallskip
{\bf keywords}: Potential reconstruction, Schr\"{o}dinger equation, Neumann-to-Dirichlet map, modified Bessel function.

\section{Introduction}

In 1980 Calder\'{o}n published his paper entitled 'On an inverse boundary value problem' \cite{Calderon1980}. This pioneer contribution motivated many developments in inverse problems, in particular the inverse problem that we consider in this paper is also related closely with the classical Calder\'{o}n problem. In 1987, Sylvester and Uhlmann proved the uniqueness with many boundary measurements in $\mathbb{R}^3$ for $U\in C(\overline{\Omega})$. Nowadays there are many generalizations involving this result. For example, the case of reconstruction using partial boundary measurements. Moreover, one can refer to \cite{A.L.BUKHGEIM and G.UHLMANN-2002, Isakov2007, Imanuvilov2010, Imanuvilov2013} and a survey paper \cite{Kenig2014} for details. For the uniqueness results of Calder\'{o}n problem with single boundary measurement, to our knowledge, the first result is given by Isakov in 1989 \cite{Iskov1989}. Recently, Alberti and Santacesaria established the uniqueness, stability estimates and reconstruction algorithm for determining the potential in (\ref{1.1})-(\ref{1.2}) from a finite number of boundary measurements \cite{Alberti2018}. The studies of nonuniqueness are directly linked to the researches about invisibility \cite{Greenleaf2003, Greenleaf2003-1, Greenleaf2007, Greenleaf2009} and virtual reshaping \cite{Liu2009}. In fact,  Greenleaf, Lassas and Uhlmann construct some counterexamples to uniqueness in Calder\'{o}n problem by transformation optics \cite{Greenleaf2003, Greenleaf2003-1}. Furthermore, in \cite{Liu2009}, Liu also used transformation optics to reshape an obstacle in acoustic and electromagnetic scattering.

This paper is a follow-up study of our earlier work \cite{Zheng2019}, in which we studied the potential reconstruction in two-dimensional radial concentric core-shell structure. In present work, we want to further investigate the more general situation in which the domain we consider is 2-D non-radial concentric core-shell structure and three-dimensional one. More specifically, we consider the determination of piecewise constant potential in the unit disc. Based on the analytic formula of solution and the monotonicity of modified Bessel function, we prove that only one boundary measurement can recover the potential uniquely. In addition, by choosing appropriate potential and radius of the core for different core-shell structure, the boundary data $\psi|_{\partial\Omega}$ can be concordant. In other words, one boundary data are not able to determine the shape of core and potential simultaneously.

We now describe more precisely the mathematical problem. Let $\Omega\subseteq \mathbb{R}^{n}$ (n=2 or 3) be an open bounded domain containing origin possibly with multi-layered structure with a smooth boundary $\partial\Omega$, and $\upsilon$ be the outward unit normal vector to $\partial\Omega$.  Then we consider the steady state Schr\"{o}dinger equation \cite{Cohen-Tannoudji-1991} as follows
\begin{eqnarray}\label{1.1}
\left(-\frac{\hbar^2}{2m}\Delta+U(x)\right)\psi=E\psi,\ \ \ \ \ \text{in}\ \Omega,
\end{eqnarray}
with Neumann boundary condition
\begin{equation}\label{1.2}
\frac{\partial\psi}{\partial \upsilon}=g,\ \ \ \ \ \ \text{on}\ \partial\Omega,
\end{equation}
where $\hbar$, $m$ denote the reduced Planck's constant and the mass of particles respectively, $U(x)$ is the potential, $E$ is the energy value, and the solution $\psi(x)$ is called de Broglie's matter wave.

The Neumann-to-Dirichlet map is given by
\begin{equation}\label{1.2}
R_U^E:\ \frac{\partial\psi}{\partial\upsilon}\bigg|_{\partial\Omega}\ \mapsto\ \psi|_{\partial\Omega}.
\end{equation}

In this paper, we pay attention to solving the following inverse problem,

\emph{Inverse problem}: For arbitrary fixed energy value $E>0$, recover the potential $U(x)$ from one boundary measurement $\psi|_{\partial\Omega}$ knowing $R_U^E$.

This paper is organized as follows :
In section 2, we study the two-dimensional case. In particular, in section 2.1, we formulate the solution of the forward problem and the associated Neumann to Dirichlet map are given. And the uniqueness theorem of potential reconstruction from one boundary measurement is established in section 2.2. Additionally, the nonuniqueness result is obtained in section 2.3. In section 3, we deal with the case of 3-D non-radial concentric annulus, the structure of this section is same as section 2. Finally, a conclusion is given in Section 4.

\section{The potential reconstruction in 2-D core-shell structure }
\subsection{Solution formula and Neumann to Dirichlet map}
In this section, under the polar coordinates, we derive the analytic solution formula of (\ref{1.1})-(\ref{1.2}) in 2-D core-shell structure, and define the Neumann to Dirichlet map (ND map).

Multiplying each side by $-\frac{2m}{\hbar^2}$ in (\ref{1.1}), we have
\begin{equation}\label{2.1}
  \Delta\psi(x)-\widetilde{U}(x)\psi(x)=-\widetilde{E}\psi(x),
\end{equation}
where $\widetilde{U}(x)=\frac{2m}{\hbar^2}U(x)$, $\widetilde{E}=\frac{2m}{\hbar^2}E$.
Let $\Omega$ be an annulus of radius $r_1$ and $1$ (core-shell structure), and the potential $U(x)$ be a piecewise constant function, i.e.
\begin{eqnarray}\label{2.3}
\widetilde{U}(x)=
\begin{cases}
\widetilde{U}_1,\ \ \ \ |x|<r_1,\\
\widetilde{U}_2,\ \ \ \ r_1<|x|<1.
\end{cases}
\end{eqnarray}
We assume for simplicity $\widetilde{U}_1=\widetilde{E}+\sigma_1^{-1}$, $\sigma_1>0$, and $\widetilde{U}_2=\widetilde{E}+1$. Then, by the polar coordinate transformation, (\ref{1.1}) becomes
\begin{eqnarray}\label{2.4}
\frac{1}{r}\frac{\partial}{\partial r}\left(r\frac{\partial\psi}{\partial r}\right)+\frac{1}{r^{2}}\frac{\partial^{2}\psi}{\partial \phi
^{2}}-\left(\sigma_1^{-1}\chi_{\{r<r_1\}}+1\chi_{\{r_1<r<1\}}\right)\psi=0,\ \ \ \ \ r\in(0,1).
\end{eqnarray}
Suppose that $\psi|_{r=0}$ is bounded. The corresponding ND map can is given by
\begin{equation}\label{2.5}
R_{\sigma_1,r_1}:\ H^{-\frac{1}{2}}(\partial \Omega)\ni\frac{\partial\psi}{\partial r}\bigg|_{r=1}\ \mapsto\ \psi|_{r=1}\in H^{\frac{1}{2}}(\partial \Omega).
\end{equation}
Hence, the reconstruction of the potential $U(x)$ is reduced to recovering the $\sigma_1$.

Furthermore, setting $\frac{\partial\psi}{\partial r}\big|_{r=1}=g(\phi)$, the ND map (\ref{2.5}) can be represented
by solving problem (\ref{2.4}) with Neumann boundary condition $\frac{\partial\psi}{\partial r}|_{r=1}=g(\phi)$:
\begin{eqnarray}\label{2.6}
\begin{cases}
\frac{1}{r}\frac{\partial}{\partial r}\left(r\frac{\partial\psi}{\partial r}\right)+\frac{1}{r^{2}}\frac{\partial^{2}\psi}{\partial \phi
^{2}}-\sigma_1^{-1}\psi=0,\ \ \ \ \ r\in(0,r_1), \phi\in(0,2\pi),\\
\frac{1}{r}\frac{\partial}{\partial r}\left(r\frac{\partial\psi}{\partial r}\right)+\frac{1}{r^{2}}\frac{\partial^{2}\psi}{\partial \phi^{2}}-\psi=0,\ \ \ \ \ r\in(r_1,1), \phi\in(0,2\pi),\\
\psi|_{r=r_1}^+=\psi|_{r=r_1}^-,\\
\frac{\partial\psi}{\partial r}\big|_{r=r_1}^+=\sigma_1\frac{\partial\psi}{\partial r}\big|_{r=r_1}^-,\\
\frac{\partial\psi}{\partial r}|_{r=1}=g(\phi),\\
\psi|_{r=0}\ \ \text{is bounded}.
\end{cases}
\end{eqnarray}
where $\cdot|_{r=r_1}^+$ means the limit to the outside of $\{r|r=r_1\}$ and $\cdot|_{r=r_1}^-$ means the limit to the inside of $\{r|r=r_1\}$.

By the boundedness of $\psi(0)$, we assume that the matter wave $\psi(r,\phi)$ has the form
\begin{eqnarray}\label{2.7}
\psi(r,\phi)=
\begin{cases}
 \sum\limits_{n=-\infty}^{\infty}u_nI_n\left(\frac{r}{\sqrt{\sigma_1}}\right)e^{in\phi},\ \ \ \ \ r\in(0,r_1),\\
\sum\limits_{n=-\infty}^{\infty}(v_nI_n\left(r\right)+w_nK_n\left(r\right))e^{in\phi}
,\ \ \ \ \ r\in(r_1,1),
\end{cases}
\end{eqnarray}
where $I_n\left(r\right)$ and $K_n\left(r\right)$ $(n\in\mathbb{N})$ denote the $n$-th order modified Bessel functions of the first and the second kind, respectively.
$u_n$, $v_n$, $w_n$ are unknown coefficients.

From the transmission conditions on the interface $\{r|r=r_1\}$ and boundary value condition on $\{r|r=1\}$, we obtain that
\begin{eqnarray}\label{2.8}
\begin{cases}
u_nI_n\left(\frac{r_1}{\sqrt{\sigma_1}}\right)=v_nI_n\left(r_1\right)+w_nK_n\left(r_1\right),\\
u_n\sigma_1I_n'\left(\frac{r_1}{\sqrt{\sigma_1}}\right)=v_nI_n'\left(r_1\right)+w_nK_n'\left(r_1\right),\\
v_nI_n'\left(1\right)+w_nK_n'\left(1\right)=g_n,
\end{cases}
\end{eqnarray}
here $g_n=(g(\theta),e^{in\theta})$ is Fourier coefficient.

By solving (\ref{2.8}), we have
\begin{eqnarray}\label{2.9}
\begin{cases}
u_n=\frac{\left(\rho(r_1,\sigma_1)K_n(r_1)-I_n(r_1)\right)g_n}{\left(\rho(r_1,\sigma_1)K_n'(1)-I_n'(1)\right)I_n
\left(\frac{r_1}{\sqrt{\sigma_1}}\right)},\\
v_n=-\frac{g_n}{\rho(r_1,\sigma_1)K_n'(1)-I_n'(1)},\\
w_n=\frac{\rho(r_1,\sigma_1)g_n}{\rho(r_1,\sigma_1)K_n'(1)-I_n'(1)},
\end{cases}
\end{eqnarray}
where
\begin{eqnarray}\label{2.10}
\rho(r_1,\sigma_1)=\frac{\sigma_1I_n'\left(\frac{r_1}{\sqrt{\sigma_1}}\right)I_n\left(r_1\right)-I_n\left(\frac{r_1}
{\sqrt{\sigma_1}}\right)I_n'\left(r_1\right)}{\sigma_1I_n'\left(\frac{r_1}{\sqrt{\sigma_1}}\right)K_n\left(r_1\right)-I_n\left(\frac{r_1}
{\sqrt{\sigma_1}}\right)K_n'\left(r_1\right)}.
\end{eqnarray}
Hence, by substituting the coefficient formula (\ref{2.9}) into (\ref{2.7}), we get the solution of problem (\ref{2.6}). Finally, the ND map can be expressed precisely as follows,
\begin{equation}\label{2.11}
R_{\sigma_1,r_1}(g)=\sum\limits_{n=-\infty}^{\infty}\frac{\rho(r_1,\sigma_1)K_n(1)-I_n(1)}{\rho(r_1,\sigma_1)K_n'(1)-I_n'(1)}g_ne^{in\phi}.
\end{equation}
Clearly, $R_{\sigma_1,r_1}:\ H^{-\frac{1}{2}}(\partial \Omega)\rightarrow H^{\frac{1}{2}}(\partial \Omega)$ is a multiplier operator, and its operator norm is defined by
\begin{equation}\label{2.111}
\|R_{\sigma_1,r_1}\|=\sup_{g\in {H^{-\frac{1}{2}}(\partial \Omega)}}\frac{\|R_{\sigma_1,r_1}(g)\|_{H^{\frac{1}{2}}(\partial \Omega)}}{\|g\|_{H^{-\frac{1}{2}}(\partial \Omega)}}
\end{equation}
From (\ref{2.11}), it implies
\begin{equation}\label{2.112}
\|R_{\sigma_1,r_1}\|=\sup_{g\in H^{-\frac{1}{2}}(\partial \Omega)}\frac{(\sum\limits_{n=-\infty}^{\infty}(1+|n|^2)^{\frac{1}{2}}\big|\frac{\rho(r_1,\sigma_1)K_n(1)-I_n(1)}{\rho(r_1,\sigma_1)K_n'(1)-I_n'(1)}
\big|^2|g_n|^2)^{\frac{1}{2}}}{(\sum\limits_{n=-\infty}^{\infty}(1+|n|^2)^{-\frac{1}{2}}|g_n|^2)^{\frac{1}{2}}}
\end{equation}

Next, we define the following ND map for Schr\"{o}dinger equation in a disk:
\begin{equation}\label{2.12}
R(g)=\Psi|_{r=1},
\end{equation}
where $\Psi$ is the solution to
\begin{eqnarray}\label{2.13}
\begin{cases}
r^{-1}\frac{\partial}{\partial r}\left(r\frac{\partial\Psi}{\partial r}\right)+\frac{1}{r^{2}}\frac{\partial^{2}\Psi}{\partial \phi
^{2}}-\Psi=0,\ \ \ \ \ r\in(0,1),\\
\frac{\partial\Psi}{\partial r}|_{r=1}=g(\phi),\\
\Psi|_{r=0}\ \ \text{is bounded}.
\end{cases}
\end{eqnarray}
Similarly, the ND map (\ref{2.12}) can be also represented by
\begin{equation}\label{2.14}
R(g)=\sum\limits_{n=-\infty}^{\infty}\frac{I_n(1)}{I_n'(1)}g_ne^{in\phi}.
\end{equation}

\subsection{Uniqueness}
In this section, the main result is that we establish the uniqueness theorem for the potential reconstruction problem from one boundary measurement, i.e., determining uniquely the piecewise constant potential in the 2-D core-shell structure by $\psi|_{r=1}$. We first give the asymptotic property of ND map with respect to the radius of core $r_1$ and potential coefficient $\sigma_1$. Next, for proving the uniqueness result, we introduce some important lemmas about modified bessel function. Finally£¬we establish the uniqueness theorem by above lemmas.
\begin{thm}\label{th2.1}
Let $R_{\sigma_1,r_1}$, $R$ are defined by (\ref{2.5}) and (\ref{2.12}) respectively. Then,\\
(1)\ for any fixed $r_1\in(0,1)$, we have that
\begin{equation}\label{2.15}
\|R_{\sigma_1,r_1}-R\|\rightarrow0,\ (as\ \sigma_1\rightarrow1);
\end{equation}
(2)\ for any fixed $\sigma_1>0$, we have that
\begin{equation}\label{2.151}
\|R_{\sigma_1,r_1}-R\|\rightarrow0,\ (as\ r_1\rightarrow0);
\end{equation}
\end{thm}

\begin{proof}

(1)\ It is a straightforward consequence of the definitions of operator norm for the ND map $R_{\sigma_1,r_1}$ and $R$.

(2)\ Notice that the following asymptotic behavior of modified Bessel function \cite{Abramowitz1964}: For $n\geq1$, we have
\begin{eqnarray}\label{2.16}
\begin{cases}
I_n(r)=\frac{1}{2^{n}n!}r^n+\frac{1}{2^{n+2}(n+1)!}r^{n+2}+O(r^{n+3});\\
I_n'(r)=\frac{1}{2^{n}(n-1)!}r^{n-1}+\frac{n+2}{2^{n+2}(n+1)!}r^{n+1}+O(r^{n+2});\\
K_n(r)=\frac{2^{n-1}(n-1)!}{r^{n}}-\frac{2^{n-3}(n-2)!}{r^{n-2}}+O(\frac{1}{r^{n-3}});\\
K_n'(r)=-\frac{2^{n-1}n!}{r^{n+1}}+\frac{2^{n-3}(n-2)(n-2)!}{r^{n-1}}+O(\frac{1}{r^{n-2}}).\\
\end{cases}
\end{eqnarray}
For $n=0$, we have
\begin{eqnarray}\label{n0}
\begin{cases}
I_0(r)=1+\frac{\frac{1}{4}r^2}{(1!)^2}+o(r^2);\\
I_0'(r)=\frac{1}{2}r+\frac{\frac{1}{4}r^3}{(2!)^2}+o(r^3);\\
K_0(r)=-\{\ln(\frac{1}{2}r)+\gamma\}I_0(r)+\frac{\frac{1}{4}r^2}{(1!)^2}+o(r^2);\\
K_0'(r)=-\frac{1}{r}I_0(r)-\{\ln(\frac{1}{2}r)+\gamma\}I_1(r)-o(r).\\
\end{cases}
\end{eqnarray}
A combination of (\ref{2.16}-\ref{n0}) and the definitions of operator norm for the ND map yields (\ref{2.151}).
\end{proof}

For the uniqueness, we will need the following important results from the theory of modified bessel function. The following lemma was given in \cite{Diego2016} (page 1236).
\begin{lem}\label{lem1}
For $\nu\geq0, \alpha\geq1$ and $x>0$ the following holds:
\begin{equation}\label{2.17}
\frac{I_{\nu+1}(x)}{I_\nu(x)}>\frac{x}{\lambda+\sqrt{\lambda^2+x^2}}, \lambda=\nu+1+\frac{1}{2}(\alpha-1)
\end{equation}
\end{lem}

\begin{lem}\label{lem2}
For $\eta>0, \alpha\geq1$ and $r>0$ the function:
\begin{equation}\label{2.18}
F(\eta)=\eta^\alpha\frac{I_\nu'(\eta^{-1}r)}{I_\nu(\eta^{-1}r)},
\end{equation}
is strictly is monotonically increasing with respect to $\eta$.
\end{lem}

\begin{proof}
Firstly, changing variables $x=\eta^{-1}r$, then $F(\eta)=\eta^\alpha\frac{I_\nu'(\eta^{-1}r)}{I_\nu(\eta^{-1}r)}$ is converted to
\begin{equation}\label{2.181}
H(x)=r^{\alpha}x^{-\alpha}\frac{I_\nu'(x)}{I_\nu(x)}=r^{\alpha}h(x),
\end{equation}
here $h(x)=x^{-\alpha}\frac{I_\nu'(x)}{I_\nu(x)}$, hence we only need to prove $h(x)$ is monotone with respect to $x$.\\
Next by $xI_{\nu}'(x)-\nu I_{\nu}(x)=xI_{\nu+1}(x)$, we derive that
\begin{equation}\label{2.19}
h(x)=\frac{\nu}{x^{\alpha+1}}+x^{-\alpha}\frac{I_{\nu+1}(x)}{I_{\nu}(x)}.
\end{equation}
Again set $h_0(x)=x^{-\alpha}\frac{I_{\nu+1}(x)}{I_\nu(x)}$, by direct calculation using $xI_{\nu}'(x)+\nu I_{\nu}(x)=xI_{\nu-1}(x)$ and $xI_{\nu}'(x)-\nu I_{\nu}(x)=xI_{\nu+1}(x)$, we can obtain that
\begin{equation}\label{2.191}
h_0'(x)=x^{-\alpha}(1-\frac{2\lambda}{x}\frac{I_{\nu+1}(x)}{I_{\nu}(x)}-(\frac{I_{\nu+1}(x)}{I_{\nu}(x)})^2),
\end{equation}
Finally, let $h_1(x)=1-\frac{2\lambda}{x}\frac{I_{\nu+1}(x)}{I_\nu(x)}-(\frac{I_{\nu+1}(x)}{I_{\nu}(x)})^2$, using lemma \ref{lem1} we have
\begin{align*}\label{2.192}
h_1(x)&=1-\frac{2\lambda}{x}\frac{I_{\nu+1}(x)}{I_{\nu}(x)}-(\frac{I_{\nu+1}(x)}{I_{\nu}(x)})^2\\
&<1-\frac{2\lambda}{x}\frac{x}{\lambda+\sqrt{\lambda^2+x^2}}-(\frac{x}{\lambda+\sqrt{\lambda^2+x^2}})^2\\
&<\frac{1}{(\lambda+\sqrt{\lambda^2}+x^2)^2}((\lambda+\sqrt{\lambda^2+x^2})^2-2\lambda(\lambda+\sqrt{\lambda^2+x^2})-x^2)=0,
\end{align*}
So, $h_0'(x)<0$, then $h_0(x)$ is monotonically decreasing. Moreover, $\frac{\nu}{x^{\alpha+1}}$ is also monotonically decreasing, hence $h(x)$ is monotonically decreasing so that we have $H(x)$ is monotonically decreasing. However, notice the monotonicity of $F(\eta)$ with respect to $\eta$ is opposite to that of $H(x)$ on $x$, therefore $F(\eta)$ is strictly monotonically increasing with respect to $\eta$ as asserted.
\end{proof}

The following corollaries are quite important for proving the uniqueness, which are used in section (2) and section (3), respectively.
\begin{cor}\label{cor1}
For $\nu=n$, $\eta>0$ and $\alpha=2$, the function:
\begin{equation}\label{2.182}
F(\eta)=\eta^2\frac{I_n'(\eta^{-1}r)}{I_n(\eta^{-1}r)}
\end{equation}
is strictly monotonically increasing with respect to $\eta$.
\end{cor}

\begin{cor}\label{cor2}
For $\nu=n+\frac{1}{2}$, $\eta>0$ and $\alpha=2$, the function:
\begin{equation}\label{2.183}
F(\eta)=\eta^2\frac{I_{n+\frac{1}{2}}'(\eta^{-1}r)}{I_{n+\frac{1}{2}}(\eta^{-1}r)}
\end{equation}
is strictly monotonically increasing with respect to $\eta$.
\end{cor}

In order to get the uniqueness and nonuniqueness for piecewise constant potential reconstruction in 2-D core-shell structure, we introduce the notations from \cite{Jaeger1941}:
\begin{align}
&D(x,y)=I_\nu(x)K_\nu(y)-K_\nu(x)I_\nu(y)\label{2.23},\\
&D_{r,s}(x,y)=\frac{\partial^{r+s}}{\partial x^r\partial y^s}D(x,y)\label{2.24}.
\end{align}

The following properties are trivial, and also from \cite{Jaeger1941}.
\begin{align}
&D_{1,0}(x,x)=x^{-1},\label{2.28}\\
&D_{0,1}(x,y)=-D_{1,0}(y,x),\\
&D(x,y)D_{1,0}(x,z)-D(x,z)D_{1,0}(x,y)=x^{-1}D(z,y),\\
&D(x,y)D_{1,1}(x,z)-D_{0,1}(x,z)D_{1,0}(x,y)=x^{-1}D_{1,0}(z,y),\\
&D_{1,1}(x,y)D_{0,1}(x,z)-D_{0,1}(x,y)D_{1,1}(x,z)=-x^{-1}D_{1,1}(z,y)\label{2.32}
\end{align}

Consider $\nu=n$ in (\ref{2.23})-(\ref{2.24}), then the ND map can be rewritten as
\begin{equation}\label{2.27}
R_{\sigma_1,r_1}(g)=\sum\limits_{n=-\infty}^{\infty}\frac{I_n\left(\frac{r_1}{\sqrt{\sigma_1}}\right)D_{0,1}(1,r_1)-\sigma_1I_n'\left(\frac{r_1}
{\sqrt{\sigma_1}}\right)D(1,r_1)}{I_n\left(\frac{r_1}{\sqrt{\sigma_1}}\right)D_{1,1}(1,r_1)-\sigma_1I_n'\left(\frac{r_1}
{\sqrt{\sigma_1}}\right)D_{1,0}(1,r_1)}g_ne^{in\phi}.
\end{equation}

Then, by using Corollary \ref{cor1}, we can prove that the following uniqueness result holds.
\begin{thm}\label{th2.2}
(Uniqueness) For arbitrary fixed $r_1\in(0,1)$, and any $\sigma_j>0$, $j=1,2$, assume that $R_{\sigma_1,r_1}=R_{\sigma_2,r_1}$. Then,
\begin{equation}\label{2.33}
\sigma_1=\sigma_2.
\end{equation}
\end{thm}

\begin{proof}
Since $R_{\sigma_1,r_1}=R_{\sigma_2,r_1}$, by expression formula (\ref{2.27}), we find
\begin{align*}
&\frac{I_n\left(\frac{r_1}{\sqrt{\sigma_1}}\right)D_{0,1}(1,r_1)-\sigma_1I_n'\left(\frac{r_1}
{\sqrt{\sigma_1}}\right)D(1,r_1)}{I_n\left(\frac{r_1}{\sqrt{\sigma_1}}\right)D_{1,1}(1,r_1)-\sigma_1I_n'\left(\frac{r_1}
{\sqrt{\sigma_1}}\right)D_{1,0}(1,r_1)}\\
=&\frac{I_n\left(\frac{r_1}{\sqrt{\sigma_2}}\right)D_{0,1}(1,r_1)-\sigma_2I_n'\left(\frac{r_1}
{\sqrt{\sigma_2}}\right)D(1,r_1)}{I_n\left(\frac{r_1}{\sqrt{\sigma_2}}\right)D_{1,1}(1,r_1)-\sigma_2I_n'\left(\frac{r_1}
{\sqrt{\sigma_2}}\right)D_{1,0}(1,r_1)}.
\end{align*}
From (\ref{2.28})-(\ref{2.32}), then by straightforward calculation, we derive that
\begin{equation*}
\sigma_1\frac{I_n'(\frac{r_1}{\sqrt{\sigma_1}})}{I_n(\frac{r_1}{\sqrt{\sigma_1}})}
=\sigma_2\frac{I_n'(\frac{r_1}{\sqrt{\sigma_2}})}{I_n(\frac{r_1}{\sqrt{\sigma_2}})},
\end{equation*}
and therefore, by using monotonicity Corollary \ref{cor1}, we conclude $\sigma_1=\sigma_2$. This finishes the proof.
\end{proof}

\subsection{Nonuniqueness}
In this section, we prove the nonuniqueness of potential reconstruction problem in 2-D core-shell structure, when the radius $r_1$ and potential coefficient $\sigma_1$ satisfy some conditions.
\begin{thm}\label{th2.3}
(Non-uniqueness) Suppose that $r_j\in(0,1)$, $\sigma_j>0$, $j=1,2$, furthermore $\{r_1, \sigma_1\}$ and $\{r_2, \sigma_2\}$ satisfy
\begin{equation}\label{2.34}
D(r_1, \sigma_1, r_2, \sigma_2)=0.
\end{equation}
Then, for every $g\in H^{-\frac{1}{2}}(\partial\Omega)$,
\begin{equation}\label{2.35}
R_{\sigma_1,r_1}(g)=R_{\sigma_2,r_2}(g),
\end{equation}
where
\begin{equation}
D(r_1, \sigma_1, r_2, \sigma_2)=
\left|\begin{array}{cccc}
D_{1,0}(r_1,r_2)&\sigma_1I_n'\left(\frac{r_1}
{\sqrt{\sigma_1}}\right)&D_{1,1}(r_1,r_2)\\
I_n\left(\frac{r_2}{\sqrt{\sigma_2}}\right)&0&\sigma_2I_n'\left(\frac{r_2}
{\sqrt{\sigma_2}}\right)\\
D(r_1,r_2)&I_n\left(\frac{r_1}{\sqrt{\sigma_1}}\right)&D_{0,1}(r_1,r_2)
\end{array}\right|.
\end{equation}
\end{thm}

\begin{proof}
Similar to the proof of Theorem \ref{th2.2}. the condition (\ref{2.34}) is equivalent to
\begin{align*}
&\frac{I_n\left(\frac{r_1}{\sqrt{\sigma_1}}\right)D_{0,1}(1,r_1)-\sigma_1I_n'\left(\frac{r_1}
{\sqrt{\sigma_1}}\right)D(1,r_1)}{I_n\left(\frac{r_1}{\sqrt{\sigma_1}}\right)D_{1,1}(1,r_1)-\sigma_1I_n'\left(\frac{r_1}
{\sqrt{\sigma_1}}\right)D_{1,0}(1,r_1)}\\
=&\frac{I_n\left(\frac{r_2}{\sqrt{\sigma_2}}\right)D_{0,1}(1,r_2)-\sigma_2I_n'\left(\frac{r_2}
{\sqrt{\sigma_2}}\right)D(1,r_2)}{I_n\left(\frac{r_2}{\sqrt{\sigma_2}}\right)D_{1,1}(1,r_2)-\sigma_2I_n'\left(\frac{r_2}
{\sqrt{\sigma_2}}\right)D_{1,0}(1,r_2)}.
\end{align*}
Here we have used the properties (\ref{2.28})-(\ref{2.32}). Then, by the definition of ND map, it follows that
\begin{equation*}
R_{\sigma_1,r_1}(g)=R_{\sigma_2,r_2}(g).
\end{equation*}

\end{proof}

\begin{rem}
In fact, (\ref{2.34}) is the sufficient and necessary condition for equation $R_{\sigma_1,r_1}(g)=R_{\sigma_2,r_2}(g)$, here $g\in H^{-\frac{1}{2}}(\partial\Omega)$.
\end{rem}

\section{The potential reconstruction in 3-D core-shell structure}
\subsection{Solution formula and Neumann to Dirichlet map}

In this section, we consider the more general case when $\Omega$ represents 3-D core-shell structure. Similarly to the 2-D case, here the potential we consider is also a piecewise constant function, but the space is in dimension 3. Based on the polar coordinate transformation, we deduce the exact solution formula of (\ref{1.1})-(\ref{1.2}) in 3-D core-shell structure, and define the Neumann to Dirichlet map (ND map).

Similarly to the 2-D case, under the polar coordinates, (\ref{1.1}) in 3-D core-shell structure becomes
\begin{equation}\label{3.4}
\frac{1}{r^2}\frac{\partial}{\partial r}\left(r^2\frac{\partial\psi}{\partial r}\right)+\frac{1}{r^{2}\sin\theta}\frac{\partial}{\partial \theta}\left(\sin\theta\frac{\partial\psi}{\partial \theta}\right)+\frac{1}{r^{2}\sin^{2}\theta}\frac{\partial^{2}\psi}{\partial \phi
^{2}}-\left(\sigma_1^{-1}\chi_{\{r<r_1\}}+1\chi_{\{r_1<r<1\}}\right)\psi=0. \ \ \ \
\end{equation}
where $r\in(0,1)$, $\theta\in(0,\pi)$, and $\phi\in(0,2\pi)$.

Setting $\frac{\partial\psi}{\partial r}\big|_{r=1}=g(\theta,\phi)$, similarly to (\ref{2.5}) the ND map in three dimensions can be represented
by solving problem (\ref{3.4}) with Neumann boundary condition $\frac{\partial\psi}{\partial r}\big|_{r=1}=g(\theta,\phi)$:
\begin{eqnarray}\label{3.6}
\begin{cases}
\frac{1}{r^2}\frac{\partial}{\partial r}\left(r^2\frac{\partial\psi}{\partial r}\right)+\frac{1}{r^{2}\sin\theta}\frac{\partial}{\partial \theta}\left(\sin\theta\frac{\partial\psi}{\partial \theta}\right)+\frac{1}{r^{2}\sin^{2}\theta}\frac{\partial^{2}\psi}{\partial \phi
^{2}}-\sigma_1^{-1}\psi=0,\ \ r\in(0,r_1),\\
\frac{1}{r^2}\frac{\partial}{\partial r}\left(r^2\frac{\partial\psi}{\partial r}\right)+\frac{1}{r^{2}\sin\theta}\frac{\partial}{\partial \theta}\left(\sin\theta\frac{\partial\psi}{\partial \theta}\right)+\frac{1}{r^{2}\sin^{2}\theta}\frac{\partial^{2}\psi}{\partial \phi
^{2}}-\psi=0,\ \ r\in(r_1,1),\\
\psi|_{r=r_1}^+=\psi|_{r=r_1}^-,\\
\frac{\partial\psi}{\partial r}\big|_{r=r_1}^+=\sigma_1\frac{\partial\psi}{\partial r}\big|_{r=r_1}^-,\\
\frac{\partial\psi}{\partial r}|_{r=1}=g(\theta,\phi),\\
\psi|_{r=0}\ \ \text{is bounded}.
\end{cases}
\end{eqnarray}
where $\cdot|_{r=r_1}^+$ means the limit to the outside of $\{r|r=r_1\}$ and $\cdot|_{r=r_1}^-$ means the limit to the inside of $\{r|r=r_1\}$.

To aid our solution of equation (\ref{3.6}), it is convenient to define a new independent variable $\mu$ as $\mu=\cos\theta$, where the domain of $\mu$ is given by $-1\leq\mu\leq1$, which maps to the variable ¦È over the corresponding domain $0\leq\theta\leq\pi$. With this change, the equations (\ref{3.6}) becomes
\begin{eqnarray}\label{3.1}
\begin{cases}
\frac{1}{r^2}\frac{\partial}{\partial r}\left(r^2\frac{\partial\psi}{\partial r}\right)+\frac{1}{r^{2}}\frac{\partial}{\partial \mu}\left((1-\mu^2)\frac{\partial\psi}{\partial \mu}\right)+\frac{1}{r^{2}(1-\mu^2)}\frac{\partial^{2}\psi}{\partial \phi
^{2}}-\sigma_1^{-1}\psi=0,\ \ r\in(0,r_1),\\
\frac{1}{r^2}\frac{\partial}{\partial r}\left(r^2\frac{\partial\psi}{\partial r}\right)+\frac{1}{r^{2}}\frac{\partial}{\partial \mu}\left((1-\mu^2)\frac{\partial\psi}{\partial \mu}\right)+\frac{1}{r^{2}(1-\mu^2)}\frac{\partial^{2}\psi}{\partial \phi
^{2}}-\psi=0,\ \ r\in(r_1,1),\\
\psi|_{r=r_1}^+=\psi|_{r=r_1}^-,\\
\frac{\partial\psi}{\partial r}\big|_{r=r_1}^+=\sigma_1\frac{\partial\psi}{\partial r}\big|_{r=r_1}^-,\\
\frac{\partial\psi}{\partial r}|_{r=1}=g(\theta,\phi),\\
\psi(\mu\rightarrow\pm1)\ \ \text{is bounded},\\
\psi|_{r=0}\ \ \text{is bounded}.
\end{cases}
\end{eqnarray}

By the boundedness of $\psi(0)$ and $\psi(\mu\rightarrow\pm1)$, we can suppose that the matter wave $\psi(r,\theta,\phi)$ has the form
\begin{eqnarray}\label{3.2}
\psi(r,\theta,\phi)=
\begin{cases}
\sum\limits_{n=0}^{\infty}\sum\limits_{m=-n}^{n}u_{nm}r^{-\frac{1}{2}}I_{n+\frac{1}{2}}\left(\frac{r}{\sqrt{\sigma_1}}\right)
P_{n}^{|m|}(\mu)e^{im\phi},\ \ \ \ \ r\in(0,r_1),\\
\sum\limits_{n=0}^{\infty}\sum\limits_{m=-n}^{n}(v_{nm}r^{-\frac{1}{2}}I_{n+\frac{1}{2}}\left(r\right)+w_{nm}r^{-\frac{1}{2}}
K_{n+\frac{1}{2}}\left(r\right))P_{n}^{|m|}(\mu)e^{im\phi}
,\ \ \ \ \ r\in(r_1,1),
\end{cases}
\end{eqnarray}
where $I_{n+\frac{1}{2}}\left(r\right)$ and $K_{n+\frac{1}{2}}\left(r\right)$ $(n\in\mathbb{N})$ denote the $(n+\frac{1}{2})$-th order modified Bessel functions of the first and the second kind, respectively.
$u_{nm}$, $v_{nm}$, $w_{nm}$ are unknown coefficients.

From the transmission conditions on the interface $\{r|r=r_1\}$ and boundary value condition on $\{r|r=1\}$, we have that
\begin{eqnarray}\label{3.8}
\begin{cases}
u_{nm}I_{n+\frac{1}{2}}\left(\frac{r_1}{\sqrt{\sigma_1}}\right)=v_{nm}I_{n+\frac{1}{2}}\left(r_1\right)+w_{nm}K_{n+\frac{1}{2}}\left(r_1\right),\\
u_{nm}\sigma_1I_{n+\frac{1}{2}}'\left(\frac{r_1}{\sqrt{\sigma_1}}\right)=v_{nm}I_{n+\frac{1}{2}}'\left(r_1\right)+w_{nm}K_{n+\frac{1}{2}}'\left(r_1\right),\\
v_{nm}I_{n+\frac{1}{2}}'\left(1\right)+w_{nm}K_{n+\frac{1}{2}}'\left(1\right)=g_{nm},
\end{cases}
\end{eqnarray}
here $g_{nm}=(g(\theta,\phi),P_{n}^{|m|}(\mu)e^{im\phi})$ is Fourier coefficient.

By solving (\ref{3.8}), we can obtain
\begin{eqnarray}\label{3.9}
\begin{cases}
u_{nm}=\frac{\left(\rho(r_1,\sigma_1)K_{n+\frac{1}{2}}(r_1)-I_{n+\frac{1}{2}}(r_1)\right)g_{nm}}{\left(\rho(r_1,\sigma_1)K_{n+\frac{1}{2}}'(1)-I_{n+\frac{1}{2}}'(1)\right)I_{n+\frac{1}{2}}
\left(\frac{r_1}{\sqrt{\sigma_1}}\right)},\\
v_{nm}=-\frac{g_{nm}}{\rho(r_1,\sigma_1)K_{n+\frac{1}{2}}'(1)-I_{n+\frac{1}{2}}'(1)},\\
w_{nm}=\frac{\rho(r_1,\sigma_1)g_{nm}}{\rho(r_1,\sigma_1)K_{n+\frac{1}{2}}'(1)-I_{n+\frac{1}{2}}'(1)},
\end{cases}
\end{eqnarray}
where
\begin{eqnarray}\label{3.10}
\rho(r_1,\sigma_1)=\frac{\sigma_1I_{n+\frac{1}{2}}'\left(\frac{r_1}{\sqrt{\sigma_1}}\right)I_{n+\frac{1}{2}}\left(r_1\right)-I_{n+\frac{1}{2}}\left(\frac{r_1}
{\sqrt{\sigma_1}}\right)I_{n+\frac{1}{2}}'\left(r_1\right)}{\sigma_1I_{n+\frac{1}{2}}'\left(\frac{r_1}{\sqrt{\sigma_1}}\right)K_{n+\frac{1}{2}}\left(r_1\right)-I_{n+\frac{1}{2}}\left(\frac{r_1}
{\sqrt{\sigma_1}}\right)K_{n+\frac{1}{2}}'\left(r_1\right)}.
\end{eqnarray}
Hence, by substituting the coefficient formula (\ref{3.9}) into (\ref{3.2}), we get the solution of problem (\ref{3.1}). Finally, the ND map can be expressed precisely as follows,
\begin{equation}\label{3.11}
R_{\sigma_1,r_1}(g)=\sum\limits_{n=0}^{\infty}\sum\limits_{m=-n}^{n}\frac{\rho(r_1,\sigma_1)K_{n+\frac{1}{2}}(1)-I_{n+\frac{1}{2}}(1)}{\rho(r_1,\sigma_1)K_{n+\frac{1}{2}}'(1)-I_{n+\frac{1}{2}}'(1)}g_{nm}P_{n}^{|m|}(\mu)e^{im\phi}
\end{equation}
Clearly, $R_{\sigma_1,r_1}:\ H^{-\frac{1}{2}}(\partial \Omega)\rightarrow H^{\frac{1}{2}}(\partial \Omega)$ is a multiplier operator, and its operator norm is defined by
\begin{equation}\label{3.111}
\|R_{\sigma_1,r_1}\|=\sup_{g\in {H^{-\frac{1}{2}}(\partial \Omega)}}\frac{\|R_{\sigma_1,r_1}(g)\|_{H^{\frac{1}{2}}(\partial \Omega)}}{\|g\|_{H^{-\frac{1}{2}}(\partial \Omega)}}
\end{equation}
From (\ref{3.11}), it implies

\begin{equation}\label{3.112}
\|R_{\sigma_1,r_1}\|=\sup_{g\in H^{-\frac{1}{2}}(\partial \Omega)}\frac{(\sum\limits_{n=0}^{\infty}\sum\limits_{m=-n}^{n}(1+|m|^2)^{\frac{1}{2}}\big|\frac{\rho(r_1,\sigma_1)K_{n+\frac{1}{2}}
(1)-I_{n+\frac{1}{2}}(1)}{\rho(r_1,\sigma_1)K_{n+\frac{1}{2}}'(1)-I_{n+\frac{1}{2}}'(1)}
\big|^2|g_{nm}|^2)^{\frac{1}{2}}}{(\sum\limits_{n=0}^{\infty}\sum\limits_{m=-n}^{n}(1+|m|^2)^{-\frac{1}{2}}|g_{nm}|^2)^{\frac{1}{2}}}
\end{equation}

Next, we define the following ND map for Schr\"{o}dinger equation in a disk:
\begin{equation}\label{3.12}
R(g)=\Psi|_{r=1},
\end{equation}
where $\Psi$ is the solution to
\begin{eqnarray}\label{3.13}
\begin{cases}
\frac{1}{r^2}\frac{\partial}{\partial r}\left(r^2\frac{\partial\psi}{\partial r}\right)+\frac{1}{r^{2}}\frac{\partial}{\partial \mu}\left((1-\mu^2)\frac{\partial\psi}{\partial \mu}\right)+\frac{1}{r^{2}(1-\mu^2)}\frac{\partial^{2}\psi}{\partial \phi
^{2}}-\psi=0,\ \ \ \ \ r\in(0,1),\\
\frac{\partial\Psi}{\partial r}|_{r=1}=g(\theta,\phi),\\
\Psi|_{r=0}\ \ \text{is bounded}.
\end{cases}
\end{eqnarray}
Similarly, the ND map (\ref{3.12}) can be also represented by
\begin{equation}\label{3.14}
R(g)=\sum\limits_{n=0}^{\infty}\sum\limits_{m=-n}^{n}\frac{I_{n+\frac{1}{2}}(1)}{I_{n+\frac{1}{2}}'(1)}g_{nm}P_{n}^{|m|}(\mu)e^{im\phi}.
\end{equation}

\subsection{Uniqueness}
As before, we mainly establish the uniqueness theorem in 3-D core-shell structure in this section. In addition, we also give the asymptotic property of ND map respect to the radius of core $r_1$ and potential coefficient $\sigma_1$.
\begin{thm}\label{th3.1}
Let $R_{\sigma_1,r_1}$, $R$ are defined by (\ref{3.11}) and (\ref{3.12}) respectively. Then,\\
(1)\ for any fixed $r_1\in(0,1)$, we have that
\begin{equation}\label{3.15}
\|R_{\sigma_1,r_1}-R\|\rightarrow0,\ (as\ \sigma_1\rightarrow1);
\end{equation}
(2)\ for any fixed $\sigma_1>0$, we have that
\begin{equation}\label{3.16}
\|R_{\sigma_1,r_1}-R\|\rightarrow0,\ (as\ r_1\rightarrow0);
\end{equation}
\end{thm}

\begin{proof}

(1)\ It is a straightforward consequence of the definitions of operator norm for the ND map $R_{\sigma_1,r_1}$ and $R$.

(2)\ Notice that the following asymptotic behavior of modified Bessel function \cite{Abramowitz1964}:
\begin{equation}\label{3.17}
\begin{cases}
I_{n+\frac{1}{2}}(r)=\frac{n!2^{n+\frac{1}{2}}}{(2n+1)!\sqrt{\pi}}r^{n+\frac{1}{2}}+\frac{(n+1)!2^{n+\frac{1}{2}}}{(2n+3)!\sqrt{\pi}}r^{n+\frac{5}{2}}+O(r^{n+\frac{9}{2}}),\ \ n\geq0,\\
I_{n+\frac{1}{2}}'(r)=\frac{n!(n+\frac{1}{2})2^{n+\frac{1}{2}}}{(2n+1)!\sqrt{\pi}}r^{n-\frac{1}{2}}+\frac{(n+1)!(n+\frac{5}{2})2^{n+\frac{1}{2}}}{(2n+3)!\sqrt{\pi}}r^{n+\frac{3}{2}}+O(r^{n+\frac{7}{2}}),\ \ n\geq0,\\
K_{n+\frac{1}{2}}(r)=
\begin{cases}
\sqrt{\frac{\pi}{2}}e^{-r}r^{-\frac{1}{2}},\ \ n=0\\
\sqrt{\frac{\pi}{2}}e^{-r}(r^{-\frac{3}{2}}+r^{-\frac{1}{2}}), \ \ n=1\\
\sqrt{\frac{\pi}{2}}e^{-r}(\frac{\prod\limits_{i=1}^{n-1}(4(n+\frac{1}{2})^2)-(2i-1)^2)}{(n-1)!8^{n-1}r^{n+\frac{1}{2}}}+\frac{\prod\limits_{i=1}^{n-1}(4(n+\frac{1}{2})^2)-(2i-1)^2)}{(n-1)!8^{n-1}r^{n-\frac{1}{2}}}+O(\frac{1}{r^{n-\frac{3}{2}}})), \ \ n>1,
\end{cases}\\
K_{n+\frac{1}{2}}'(r)=
\begin{cases}
-\sqrt{\frac{\pi}{2}}e^{-r}(\frac{1}{2}r^{-\frac{3}{2}}+r^{-\frac{1}{2}}),\ \ n=0\\
-\sqrt{\frac{\pi}{2}}e^{-r}(\frac{3}{2}r^{-\frac{5}{2}}+\frac{3}{2}r^{-\frac{3}{2}})+O(r^{-\frac{1}{2}}), \ \ n=1\\
-\sqrt{\frac{\pi}{2}}e^{-r}(\frac{\prod\limits_{i=1}^{n-1}(4(n+\frac{1}{2})^2)-(2i-1)^2)(n+\frac{1}{2})}{(n-1)!8^{n-1}r^{n+\frac{3}{2}}}+\frac{\prod\limits_{i=1}^{n-1}(4(n+\frac{1}{2})^2)-(2i-1)^2)(n+\frac{1}{2})}{(n-1)!8^{n-1}r^{n+\frac{1}{2}}}+O(\frac{1}{r^{n-\frac{1}{2}}}));
\ \ n>1,
\end{cases}
\end{cases}
\end{equation}

A combination of (\ref{3.17}) and the definitions of operator norm for the ND map yields (\ref{3.15}).
\end{proof}

In order to get the uniqueness and non-uniqueness for piecewise constant potential reconstruction 3-D core-shell structure, we consider $\nu=n+\frac{1}{2}$ in (\ref{2.23})-(\ref{2.24}), then the ND map can be rewritten as
\begin{equation}\label{3.27}
R_{\sigma_1,r_1}(g)=\sum\limits_{n=0}^{\infty}\sum\limits_{m=-n}^{n}\frac{I_{n+\frac{1}{2}}\left(\frac{r_1}{\sqrt{\sigma_1}}\right)D_{0,1}(1,r_1)-\sigma_1I_{n+\frac{1}{2}}'\left(\frac{r_1}
{\sqrt{\sigma_1}}\right)D(1,r_1)}{I_{n+\frac{1}{2}}\left(\frac{r_1}{\sqrt{\sigma_1}}\right)D_{1,1}(1,r_1)-\sigma_1I_{n+\frac{1}{2}}'\left(\frac{r_1}
{\sqrt{\sigma_1}}\right)D_{1,0}(1,r_1)}g_{nm}P_{n}^{|m|}(\mu)e^{im\phi}.
\end{equation}

Similarly to the uniqueness in two dimensions, we can have the following uniqueness in three dimensions by using Corollary \ref{cor2}.
\begin{thm}\label{th3.2}
(Uniqueness) For arbitrary fixed $r_1\in(0,1)$, and any $\sigma_j>0$, $j=1,2$, assume that $R_{\sigma_1,r_1}=R_{\sigma_2,r_1}$. Then,
\begin{equation}\label{3.33}
\sigma_1=\sigma_2.
\end{equation}
\end{thm}

\begin{proof}
Since $R_{\sigma_1,r_1}=R_{\sigma_2,r_1}$, by expression formula (\ref{3.27}), we find
\begin{align*}
&\frac{I_{n+\frac{1}{2}}\left(\frac{r_1}{\sqrt{\sigma_1}}\right)D_{0,1}(1,r_1)-\sigma_1I_{n+\frac{1}{2}}'\left(\frac{r_1}
{\sqrt{\sigma_1}}\right)D(1,r_1)}{I_{n+\frac{1}{2}}\left(\frac{r_1}{\sqrt{\sigma_1}}\right)D_{1,1}(1,r_1)-\sigma_1I_{n+\frac{1}{2}}'\left(\frac{r_1}
{\sqrt{\sigma_1}}\right)D_{1,0}(1,r_1)}\\
=&\frac{I_{n+\frac{1}{2}}\left(\frac{r_1}{\sqrt{\sigma_2}}\right)D_{0,1}(1,r_1)-\sigma_2I_{n+\frac{1}{2}}'\left(\frac{r_1}
{\sqrt{\sigma_2}}\right)D(1,r_1)}{I_{n+\frac{1}{2}}\left(\frac{r_1}{\sqrt{\sigma_2}}\right)D_{1,1}(1,r_1)-\sigma_2I_{n+\frac{1}{2}}'\left(\frac{r_1}
{\sqrt{\sigma_2}}\right)D_{1,0}(1,r_1)}.
\end{align*}
From (\ref{2.28})-(\ref{2.32}), then by straightforward calculation, we derive that
\begin{equation*}\label{}
\sigma_1\frac{I_{n+\frac{1}{2}}'(\frac{r_1}{\sqrt{\sigma_1}})}{I_{n+\frac{1}{2}}(\frac{r_1}{\sqrt{\sigma_1}})}
=\sigma_2\frac{I_{n+\frac{1}{2}}'(\frac{r_1}{\sqrt{\sigma_2}})}{I_{n+\frac{1}{2}}(\frac{r_1}{\sqrt{\sigma_2}})},
\end{equation*}
and therefore, by using monotonicity Corollary \ref{cor2}, it deduces $\sigma_1=\sigma_2$.
\end{proof}

\subsection{Non-uniqueness}

In this section, it is shown that the reconstruction result in 3-D core-shell structure is not unique when the radius $r_1$ and potential coefficient $\sigma_1$ satisfy some conditions.
\begin{thm}\label{th3.3}
(Non-uniqueness) Assume that $r_j\in(0,1)$, $\sigma_j>0$, $j=1,2$, furthermore $\{r_1, \sigma_1\}$ and $\{r_2, \sigma_2\}$ satisfy
\begin{equation}\label{3.34}
D(r_1, \sigma_1, r_2, \sigma_2)=0.
\end{equation}
Then, for every $g\in H^{-\frac{1}{2}}(\partial\Omega)$,
\begin{equation}\label{3.35}
R_{\sigma_1,r_1}(g)=R_{\sigma_2,r_2}(g),
\end{equation}
where
\begin{equation}
D(r_1, \sigma_1, r_2, \sigma_2)=
\left|\begin{array}{cccc}
D_{1,0}(r_1,r_2)&\sigma_1I_{n+\frac{1}{2}}'\left(\frac{r_1}
{\sqrt{\sigma_1}}\right)&D_{1,1}(r_1,r_2)\\
I_{n+\frac{1}{2}}\left(\frac{r_2}{\sqrt{\sigma_2}}\right)&0&\sigma_2I_{n+\frac{1}{2}}'\left(\frac{r_2}
{\sqrt{\sigma_2}}\right)\\
D(r_1,r_2)&I_{n+\frac{1}{2}}\left(\frac{r_1}{\sqrt{\sigma_1}}\right)&D_{0,1}(r_1,r_2)
\end{array}\right|.
\end{equation}
\end{thm}

\begin{proof}
The proof is similar to the one in Theorem \ref{th2.2}. 

\end{proof}


\section{Conclusions}
In this article, we have discussed the inverse problem of determining the potential in 2-D and 3-D core-shell structure, given simultaneous measurements of wave function knowing the ND map. For the potential reconstruction problem,  we establish the corresponding uniqueness theorem and non-uniqueness result based the ND map and the theory of modified bessel function in 2-D and 3-D core-shell structure. The uniqueness results will be beneficial for us to reconstruct uniquely the potential and the nonuniqueness results will help us further study the question with respect to near-cloaking.

\section*{Acknowledgments}
The work described in this paper was supported by the NSF of China (11301168).

\end{document}